\documentclass[11pt]{amsart}
\usepackage{amsmath,amssymb,amsthm,mathtools}
\usepackage[left=3cm,right=3cm,top=3cm,bottom=3cm]{geometry}
\usepackage{mlmodern}
\usepackage{graphicx} 
\usepackage{tikz-cd}
\usepackage[style=alphabetic]{biblatex} 
\usepackage[dvipsnames]{xcolor}
\usepackage[colorlinks=true, citecolor=teal, linkcolor=teal]{hyperref}
\usepackage{enumitem}
\usepackage{subcaption}
\usepackage{ stmaryrd }

\def \Z {\mathbb{Z}}

\def \R {\mathbb{R}}

\def \P {\mathbb{P}}
\def \F {\mathbb{F}}
\def \RP {\mathbb{RP}}

\def \O {\mathcal{O}}
\def \H {\mathcal{H}}

\def \deg {\mathrm{deg}}
\def \dim {\mathrm{dim}}

\renewcommand{\geq}{\geqslant}
\renewcommand{\leq}{\leqslant}

\newcommand{\ie}{i.e. }

\theoremstyle{plain}
\newtheorem{thm}{Theorem}[section]

\newtheorem{cor}[thm]{Corollary}
\newtheorem{lem}[thm]{Lemma}
\newtheorem{defi}[thm]{Definition}

\theoremstyle{remark}
\newtheorem{rmk}[thm]{Remark}
\newtheorem{ex}[thm]{Example}

\makeatletter
\def\l@subsection{\@tocline{2}{0pt}{2.5pc}{5pc}{}}
\makeatother

\makeatletter
\renewcommand{\l@section}{\@tocline{1}{0pt}{10pt}{1pc}{\bfseries}}
\makeatother

\usepackage{todonotes}

\newcommand{\gadd}[1]{\textcolor{orange}{#1}}

\title{Non-existence of separating morphisms of low degree}
\author{Aloïs Demory}
\address{Eberhard Karls Universität Tübingen - Fachbereich Mathematik - Auf der Morgenstelle, 10, 72076 Tübingen, Deutschland}
\email{demory@math.uni-tuebingen.de}

\author{Gurvan Mével}
\address{Institut de Mathématiques de Jussieu -- Paris Rive Gauche, Sorbonne Université, 4 place Jussieu, 75252 Paris Cédex 5, France}
\email{gurvan.mevel@imj-prg.fr}

\author{Antoine Toussaint}
\email{antoine.toussaint@ac-creteil.fr}

\begin{document}

\begin{abstract}
    A separating curve is a real curve whose real part disconnects its complex part. It also admits totally real morphisms to the projective line, and the separating gonality is the minimal possible degree for such morphisms. The separating gonality is bounded from below by the number of connected components of the real part of the curve. In this paper, we build on a strategy of Manzaroli to improve this lower bound for some curves embedded in the projective plane or in a Hirzebruch surface.
\end{abstract}

\maketitle

\tableofcontents

\section{Introduction}\label{intro}

\subsection{Background}\label{motivation}

    The concept of \emph{separating curve} has been introduced by Klein \cite{K} in order to study the topology of real algebraic curves. A separating curve is a real curve whose real part disconnects its complex part. Although it was first introduced as a purely topological feature, Ahlfors has shown it to be equivalent to an algebraic notion, namely the existence of a \emph{separating morphism} \cite{A50}. One can then study separating curves through their separating morphisms. In particular, a separating morphism has a degree, and one can ask how do the topology and the geometry of the curve constrain the possible values for this degree. For instance, since the restriction of a separating curve to its real part is an unramified cover of $\RP^1$ then the degree is necessarily higher than the number $\ell$ of real connected components of $\mathbb R C$.

    The study of the possible degrees dates back to Ahlfors. In \cite{A50}, he gave an upper bound on the minimal degree of a separating morphism, that is $g+1$ where $g$ is the genus of the curve. This bound was later improved by Gabard \cite{G06} who showed that it can be reduced to $\frac{\ell+g+1}{2}$. 
    As shown by Coppens \cite{C11}, this result cannot be further improved in general: for any genus $g$ and for any number of components $\ell \leq g+1$ (with $\ell$ of the same parity as $g+1$), there exists a real separating curve $C$ of genus $g$ whose real part has $\ell$ connected components and that does not admit any separating morphism $f:C\to\P^1$ of degree less than $\frac{\ell+g+1}{2}$. According to the vocabulary introduced in \cite{C11}, it means that the \textit{separating gonality} of $C$ is $\frac{\ell+g+1}{2}$. In fact the result of \cite{C11} is even stronger and shows that if $g\geq2$, then for any $\max(2,\ell) \leq k \leq \frac{\ell+g+1}{2}$ there exists a real separating curve of genus $g$ with $\ell$ real components whose separating gonality is $k$.
    
    Kummer and Shaw introduced and studied the \textit{separating semi-group} of a real separating curve (\cite{KS17}), which encapsulates not only the degrees of separating morphisms but also their degree on each connected component of the real part. 
    The separating semi-groups of curves of genus 3 and 4 have been computed by Orevkov \cite{O20,O25}, and a study for plane quintics has been carried out by Magin and Orevkov \cite{MO26}. In \cite{M24b} Manzaroli focused on properties of separating $(M-2)$-curves. Magin proved that for fixed genus $g$, the set of possible separating semi-groups of a curve of genus $g$ is finite (\cite{Magin26}).
    
    So far, the main tool to obtain restrictions on the separating semi-group of a real curve is a result of Orevkov (\cite[Theorem~3.2]{O21}), which was successfully used in \cite{O25} and \cite{M24b}.
    Manzaroli also used this theorem to obtain general restrictions on the existence of separating morphisms whose degree is equal to the number $\ell$ of components of the real part in \cite{M24}. 
    Unfortunately, the proof of the main result happens to fail in some cases, as we will see in section \ref{contre-exemple}. In this paper, we build on the ideas of Manzaroli to give a corrected statement.

\subsection{Structure and results of the paper}

In section \ref{prelim} we give definitions that we use in the rest of the paper. We also recall \cite[Theorem 3.2]{O21} as our proofs rely on it.

In section \ref{contre-exemple} we present a counter-example to \cite[Theorem 1.12]{M24} and explain what the gap in the proof is.

Section \ref{sec-statements} is the main part of the paper. We first present a lemma based on Orevkov's theorem to find a lower bound on the dimension of a certain linear system. We then apply Manzaroli's strategy to the particular cases of the projective plane $\P^2$ and of the Hirzebruch surfaces $\F_n$. We also adapt the proof to study separating morphisms of degree $\ell+k$, with $k\geq0$, and not only of degree $\ell$.

In particular we obtain the following for separating curves embedded in the projective plane.

\begin{thm}[Corollary \ref{coro-P2}]
    Let $C\subset \P^2$ be an irreducible non-singular separating curve of degree $m$ whose real part has $\ell$ connected components. Let $k\geq0$ be an integer and $\varepsilon = \ell+k \mod 2$. If 
    \begin{enumerate}
        \item[(a)] either $m$ is odd and 
        \[ k^2 +2k+3-\varepsilon \leq \ell \leq  \frac{(m-3)(m+3)}{4} -k-1+\varepsilon\]

        \item[(b)] or $m$ is even and 
        \[ k^2 + 2(1-\varepsilon)k +\varepsilon +1 \leq  \ell \leq  \frac{(m-4)(m+2)}{4} -k-1+\varepsilon\]
    \end{enumerate}
    then there exists no separating morphism $f:C\to\P^1$ of degree $\ell+k$.

    In particular, if $k_{\max}$ is the maximal $k$ such that the previous hypotheses are satisfied, then the separating gonality of $C$ is at least $\ell+k_{\max}+1$.
\end{thm}

\subsection{Acknowledgments}

This work was initiated during our stay at Geneva University. We acknowledge support by the Swiss National Science Foundation grant 204125 ``Interactions of real, tropical and symplectic geometry''. AD is supported by the Walter Benjamin Programme of the Deutsche Forschungsgemeinschaft. GM is supported by ERC Grant ROGW-864919.
We would like to thank E. Brugallé, I. Itenberg, G. Mikhalkin and S. Orevkov for fruitful discussions.


\section{Preliminaries}\label{prelim}

\begin{defi}
A \emph{real algebraic variety} $X$ is a complex algebraic variety equipped with an anti-holomorphic involution $\sigma$. Its \emph{real part} is the set of fixed points of $\sigma$, \ie $\R X = \mathrm{Fix}(\sigma)$.
\end{defi}

\begin{defi}
Let $X$ be a real algebraic surface and $C \subset X$ be a non-singular real algebraic curve. An \emph{oval} of $\R C$ is a connected component of $\R C$ that bounds a disk in $\R X$.
\end{defi}

\begin{rmk}
    When $X = \mathbb P^2$ or $\mathbb F_n$ and these surfaces are endowed with their standard real structure, the disk bounded by an oval $\mathcal O$ of $\R C$ in $\R X$ is unique. In these cases, we say that another oval $\mathcal O'$ lies \emph{inside} $\mathcal O$ if $\mathcal O'$ lies in the disk bounded by $\mathcal O$.
\end{rmk}

\begin{defi}
Let $C$ be a real algebraic curve. We say that $C$ is \emph{separating} if $C\setminus \R C$ has two connected components.  
\end{defi}

\begin{ex}
    The complex projective space $\P^n$ equipped with the complex conjugation \[ \bar \cdot : \P^n \to \P^n \] is a real algebraic variety. In particular, $\P^1$ is a separating curve.
\end{ex}

\begin{defi}
Let $(X,\sigma)$ and $(Y,\rho)$ be two real algebraic varieties. A morphism $f:X\to Y$ is \emph{real} if it commutes with the involutions, \ie $f\circ \sigma = \rho\circ f$.

    Let $C$ be a real algebraic curve and $f:C\to\P^1$ be a real morphism. We say that $f$ is \emph{separating} if $f^{-1}(\RP^1) = \R C$. 
\end{defi}

Since $\P^1$ is separating, then it is immediate that if $C$ admits a separating morphism then $C$ is itself separating. As mentionned in Section \ref{motivation}, Alhfors showed that the converse is true (see \cite{A50}).

\begin{rmk}
    Let $f:C\to\P^1$ be a real morphism. By definition, for any $x\in\R C$, one has $f(x) \in \RP^1$. Hence the condition for $f$ to be separating is equivalent to ask that $f^{-1}(y) \subset \R C$ for any $y\in\RP^1$. For this reason, separating morphisms are also called \emph{totally real} morphisms.
\end{rmk}

\begin{defi}[\cite{C11}]
    The \emph{separating gonality} of a real algebraic curve $C$ is the minimal possible degree of a separating morphism $C\to\P^1$.
\end{defi}

\begin{ex}
Let $C$ be a hyperbolic real plane curve of degree 4. Its real part consists of one oval inside another. Let $p\in\R\P^2$ be a point inside the inner oval, and consider the pencil of lines $(D_t)_{t\in\P^1}$ through $p$, see figure \ref{fig-deg4}. We define
\[ \begin{array}{cccl}
      f : & C & \to & \P^1  \\
      & x & \mapsto & t \text{ such that } x\in D_t
\end{array} . \]
Then $f$ is a real morphism and one can check it is separating. This implies in particular that $C$ is also separating. Moreover, since $C\cdot D_t=4$ then the degree of $f$ is 4 and the separating gonality of $C$ is at most 4.

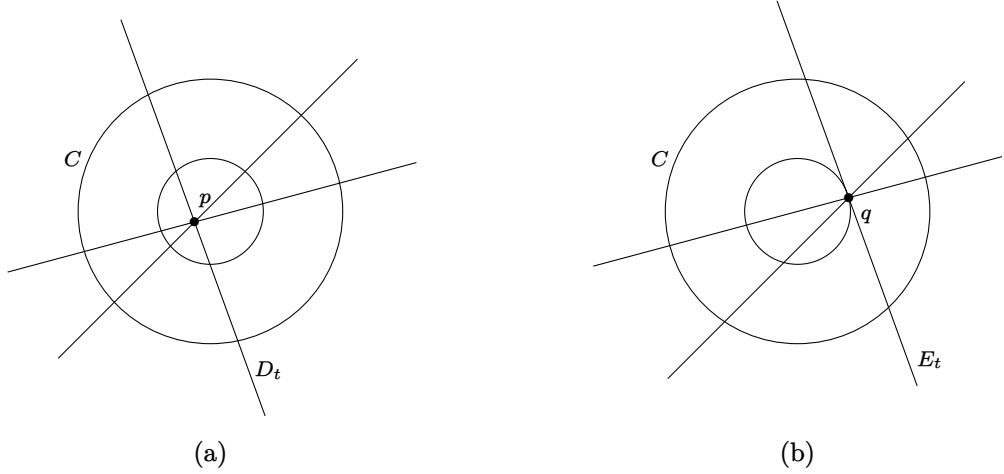
\begin{figure}[h]
\begin{subfigure}[t]{0.49\textwidth}
	\centering
	\begin{tikzpicture}[scale=0.7]

    \clip (0,0) circle (4) ;

    \draw (0,0) circle (1) ;
    \draw (0,0) circle (2.5) ;
    \node at (-2.6,1) {\scriptsize $C$} ;
    
    \node at (-0.3,-0.2) {\scriptsize $\bullet$} ;
    \node at (-0.1,0.2) {\scriptsize $p$} ;

    \draw (-0.3,-0.2) -- +(45:5cm);
    \draw (-0.3,-0.2) -- +(45:-5cm);
    \draw (-0.3,-0.2) -- +(110:5cm);
    \draw (-0.3,-0.2) -- +(110:-5cm);
    \draw (-0.3,-0.2) -- +(15:5cm);
    \draw (-0.3,-0.2) -- +(15:-5cm);
    \node at (1.1,-3) {\scriptsize $D_t$} ;
		
	\end{tikzpicture}
	\caption{}
	\label{fig-deg4}
\end{subfigure}
\begin{subfigure}[t]{0.49\textwidth}
	\centering
	\begin{tikzpicture}[scale=0.7]

    \clip (0,0) circle (4) ;

    \draw (0,0) circle (1) ;
    \draw (0,0) circle (2.5) ;
    \node at (-2.6,1) {\scriptsize $C$} ;
    
    \node at (15:1) {\scriptsize $\bullet$} ;
    \node at (1.3,-0.1) {\scriptsize $q$} ;

    \draw (15:1) -- +(45:5cm);
    \draw (15:1) -- +(45:-5cm);
    \draw (15:1) -- +(110:5cm);
    \draw (15:1) -- +(110:-5cm);
    \draw (15:1) -- +(15:5cm);
    \draw (15:1) -- +(15:-5cm);
    \node at (2.5,-2.8) {\scriptsize $E_t$} ;
    
	\end{tikzpicture}
	\caption{}
	\label{fig-deg3}
\end{subfigure}
\caption{Examples of separating morphisms.}
\end{figure}

One can also take $q\in\R\P^2$ to be on the inner oval, $(E_t)_{t\in\P^1}$ to be the pencil through $q$, see figure \ref{fig-deg3}, and 
\[\begin{array}{cccl}
      g : & C & \to & \P^1  \\
      & x\neq q & \mapsto & t \text{ such that } x\in E_t \\
      & q & \mapsto & t \text{ such that } E_t \text{ is the tangent line at } q \text{ of } C
\end{array} . \]
Then $g$ is a separating morphism of degree $3$, so the separating gonality of $C$ is at most 3. 
Since the gonality of a plane curve of degree $d$ is $d-1$, then one concludes that the separating gonality of $C$ is 3. See also \cite[Example 3.8]{KS17}.
\end{ex}

We recall \cite[Theorem 3.2]{O21}.

\begin{thm}[{\cite[Theorem 3.2]{O21}}]\label{thm-orevkov}
    Let $X$ be a smooth real algebraic surface and $C$ be a smooth irreducible real separating curve on $X$. Let $D\in|C+K_X|$ be a real divisor and assume it does not have $C$ as a component. Write $D=2D_0+D_1$ with $D_0$ effective and $D_1$ effective and reduced. 
    Fix a complex orientation of $\R C$. Fix also an orientation on $\R X \setminus(\R C \cup \R D_1)$ which changes every time we cross $\R C \cup \R D_1$ at a smooth point (this orientation comes from the canonical class $D-C\sim K_X$). This orientation induces a boundary orientation on $\R C \setminus D_1$.
    Let $f:C\to\P^1$ be a separating morphism and let $p\in\P^1$.
    
    Then either $f^{-1}(p) \setminus D = \varnothing$ or there is a point in $f^{-1}(p) \setminus D$ at which the two orientations differ.    
\end{thm}

\section{On Manzaroli's result}\label{contre-exemple}
    In this section, we present a counter-example to \cite[Theorem 1.12]{M24} and we then explain what is the gap in the proof.

    \subsection{A counterexample}

    We present a counter-example to the following.

    \begin{thm}[\textbf{\cite[Theorem 1.12]{M24}}] \label{thmmatilde}
        Let $X$ be a smooth real algebraic surface and $C \subset X$ a non-singular separating real algebraic curve 
        with $\ell$ connected components in $\R C$. Let $D \in |C+K_X|$ be a real divisor such that $C$ does not have $D$ as a component. Write $D$ as $2D_0 + D_1$, where $D_0$ is real effective and $D_1$ is a reduced curve. Assume that
        \begin{itemize}
            \item[(1)] $\chi(\O_X) > 0$;
            \item[(2)] $(-K_X)^2 \geq 0$;
            \item[(3)] $-K_X D'\geq 0$ for any effective divisor $D'$;
            \item[(4)] $D_1$ is empty;
            \item[(5)] $-K_XD_0 > 0$;
            \item[(6)] $1 - \varepsilon + \dfrac{D_0^2 + D_0K_X}{2} < \left\lfloor \dfrac{\ell}{2} \right\rfloor < \dfrac{D_0^2-D_0K_X}{2}$, where $\varepsilon = \ell \mod 2$.
        \end{itemize}
        Then there exists no separating morphism of degree $\ell$ on $C$.
    \end{thm}
    
    Let $X=\P^1\times\P^1$ endowed with its standard real structure, for which its real part is homeomorphic to a torus. 
    The canonical divisor $K_X$ is of bi-degree $(-2,-2)$. In particular, one has
\begin{itemize}
    \item[(1)] $\chi(\O_X)=1\geq1;$
    \item[(2)] $(K_X)^2=(-2)\cdot(-2)+(-2)\cdot(-2)=8\geq0$;
    \item[(3)] $-K_X\cdot D\geq0$ for all effective divisors $D$.
\end{itemize}

    We now consider a smooth real algebraic curve $C$ of bi-degree $(4,4)$, whose real part consists of $4$ disjoint connected components representing the same non-trivial class in $H_1(X(\R);\Z)$ and given by a small perturbation of a union of 4 curves with equation $(x_0+ax_1) y_1-(x_1-ax_0)y_0=0$ in homogeneous coordinates $([x_0:x_1],[y_0:y_1])$ for distinct real values of $a$, see Figure~\ref{Fig1}. The canonical divisor $K_X$ being of bi-degree $(-2,-2)$, the linear system $|C+K_X|$ consists of all curves of bi-degree $(2,2)$. We then set $D=2D_0\in|C+K_X|$ to be any real divisor whose real part does not intersect $\R C$. Hypothesis (4) is true since we have taken $D=2D_0$ and (5) is straightforward. Here, we have $\ell=4$ components of $\R C$ and so the parity $\varepsilon\in\{0,1\}$ of $\ell$ is equal to $0$. So we can compute the left-hand side in (6) to be $1-\dfrac{2-4}{2}=0$ and the right-hand side to be $\dfrac{2+4}{2}=3$. Hence hypothesis (6) requires that $0<\left\lfloor \frac{\ell}{2}\right\rfloor<3$, which is true since $\frac{\ell}{2}=2$. If the theorem were true, then it would imply that the curve $C$ does not admit any separating morphism of degree $\ell=4$. However, the two projections $\P^1\times\P^1\to\P^1$ induce by restriction separating morphisms $C\to\P^1$ and both of them are of degree $4$.

\begin{figure}[h!] 
\centering
\begin{tikzpicture}[thick,font=\footnotesize,scale=0.7]

\draw[thick](-4,-4)--(4,-4)--(4,4)--(-4,4)--(-4,-4);


\draw[ultra thick,dotted,orange](0.5,-4)node[below]{$p$}--(0.5,-3.5)node[left]{$p_1$}--(0.5,-1.5)node[left]{$p_2$}--(0.5,0.5)node[left]{$p_3$}--(0.5,2.5)node[left]{$p_4$}--(0.5,4);

\draw[blue,thick](-4,-4)--(4,4);
\draw[red,ultra thick,->](-0.01,-0.01)--(0,0);

\draw[blue,thick](-4,-2)--(2,4);
\draw[red,ultra thick,->](-1.01,0.99)--(-1,1);
\draw[blue,thick](2,-4)--(4,-2);
\draw[red,ultra thick,->](2.99,-3.01)--(3,-3);

\draw[blue,thick](-4,0)--(0,4);
\draw[red,ultra thick,->](-2.01,1.99)--(-2,2);
\draw[blue,thick](0,-4)--(4,0);
\draw[red,ultra thick,->](1.99,-2.01)--(2,-2);

\draw[blue,thick](-4,2)--(-2,4);
\draw[red,ultra thick,->](-3.01,2.99)--(-3,3);
\draw[blue,thick](-2,-4)--(4,2);
\draw[red,ultra thick,->](0.99,-1.01)--(1,-1);

\end{tikzpicture}
\caption{The real part of $C$ (in blue) in the torus $\RP^1\times\RP^1$ along with a choice of complex orientation in red and the inverse images $p_i$'s of some $p\in\RP^1$ by the separating morphism induced by $\P^1\times\P^1\to\P^1$.}
\label{Fig1}
\end{figure}
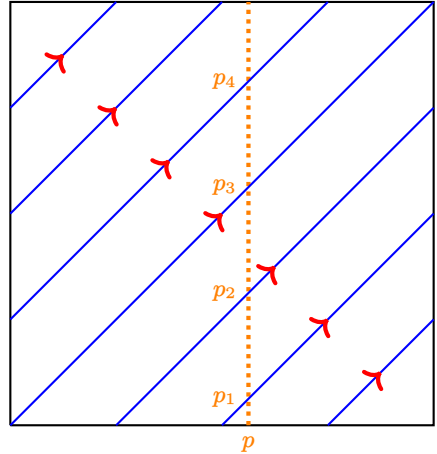

\begin{rmk}
    The exact same construction also yields a counter-example of bi-degree $(6,6)$ on the hyperboloid $X$. However, it does not generalize to bi-degrees $(2s,2s)$ for $s\geq4$. In fact, the hypothesis (6) would not be satisfied since the the left-hand side would be equal to $1-\varepsilon+\frac{2(s-1)^2-4(s-1)}{2}$ which is greater or equal than $\left\lfloor\frac{\ell}{2}\right\rfloor=s$ starting from $s=4$. This construction also provides counter-examples to Theorem 1.14 from \cite{M24} (the case where $D_1$ is non-empty) in bi-degrees $(3,3),(5,5)$ and $(7,7)$ but not higher.
\end{rmk}

\begin{rmk}
    The kind of separating morphism that we constructed is already well-known. For instance, it already appears in \cite[Example 2.13]{KS17}, and is used in the classification established in \cite{O25}.
\end{rmk}

\subsection{What is the problem?}

In the proof of \cite[Theorem 1.12]{M24}, one has to compute the self-intersection of the curves class $D_0$. This is done by exhibiting two curves $A$ and $B$ in this class which both pass through all points of the preimage $f^{-1}(p)$ and through another $h-1$ points in generic position, where $h$ is the dimension of the linear system of curves in $|D_0|$ passing through all points of $f^{-1}(p)$. The conclusion is then that $D_0^2=A\cdot B\geq \ell+h-1 = \deg(f)+h-1$ and the final result is obtained using Theorem \ref{thm-orevkov} to get a lower bound for $h$. However, if $A$ and $B$ share a common component, then one cannot compute $D_0^2$ that way. Indeed, in this counterexample one can see that the dimension of the linear system of curves in $|D_0|$ passing through all points of $f^{-1}(p)$ is $1$ but all curves in this linear system have a common component : the line $\{p\}\times\P^1$. Now, the number of connected components of $\R C$ is $\ell=4$ and the self-intersection $D_0^2$ is only 2 so the inequality $D_0^2\geq \ell+h-1$ does not hold.

\section{Corrected statements} \label{sec-statements}

\subsection{A lemma}

We start with the following lemma, which can already be found in Manzaroli's proof.

\begin{lem}\label{lem-dimD(p)}
    Let $X$ be a smooth real algebraic surface, and $C\subset X$ be a separating, irreducible and non-singular curve that admits a separating morphism $f$ of degree $\ell+k$. Let $D\in|C+K_X|$ and write $D=2D_0+D_1$ with $D_0$ effective and $D_1$ effective and reduced. Let $p\in\RP^1$. If there exists $D'_1\in|D_1|$ such that $D_1'$ passes by none of the points of the preimage $f^{-1}(p)$, then the dimension of the linear system $\mathcal{D}(p)\subset|D_0|$ of all divisors containing the preimage $f^{-1}(p)$ is at least $\dim|D_0|-\left\lfloor\dfrac{\ell+k}{2}\right\rfloor$.
\end{lem}

\begin{proof} 
    The divisor $D_1'$ induces an orientation of $\R X\setminus\left(C\cup D_1'\right)$ which, in turn, induces an orientation on $\R C\setminus D_1'$. Let $V(p)$ be the subset of $f^{-1}(p)$ where this orientation differs from a fixed complex orientation of $\R C$. Up to reversing this orientation, one can always assume that $V(p)$ contains at most $\left\lfloor\frac{\ell+k}{2}\right\rfloor$ points. By Theorem \ref{thm-orevkov}, if a curve $D_0'\in|D_0|$ passes by all the points of $V(p)$, then it passes by all the points of $f^{-1}(p)$, hence the result.
\end{proof}

\subsection{For the projective plane $\P^2$}

\begin{thm}\label{thm-P2}
    Let $C\subset \P^2$ be an irreducible non-singular separating curve of degree $m$ whose real part has $\ell$ connected components. Choose $D\in|C+K_{\P^2}|$ (\ie $D$ is a curve of degree $d=m-3$) and a decomposition $D=2D_0+D_1$ such that $D_0$ is effective (of degree $d_0$) and $D_1$ is a reduced curve (of degree $d_1$). Let $\ell_n= m \mod 2$ be the number of pseudo-lines of $\R C$ and $\ell_c=\ell-\ell_n$ be the number of ovals of $\R C$. Let $k\geq0$ be an integer and $\varepsilon = \ell+k \mod 2$. If
   \begin{enumerate}[label=(\roman*)]
       \item $\ell+k-\varepsilon < d_0^2+3d_0$
       \item $\ell_c-\ell_n-3k + \varepsilon  > d_0^2-3d_0+2$
   \end{enumerate}
    then there exists no separating morphism $f:C\to\P^1$ of degree $\ell+k$.
\end{thm}

\begin{proof} 
    Assume that there exists a separating morphism $f:C\to\P^1$ of degree $\ell+k$ satisfying~$(i)$, \ie such that $\ell+k-\varepsilon<d_0^2+3d_0$. We will show that $(ii)$ cannot be satisfied.
    
    Choose a point $p\in\RP^1$ and consider the linear system $\mathcal{D}(p)\subset|D_0|$ of effective divisors linearly equivalent to $D_0$ and passing through all the points of $f^{-1}(p)$. By Lemma \ref{lem-dimD(p)}, the dimension $h$ of $\mathcal{D}(p)$ is at least $\dim| D_0|-\dfrac{\ell+k-\varepsilon}{2}$. 
    One has $\dim|D_0| =  \dfrac{d_0^2+3d_0}{2}$ so hypothesis~$(i)$ on $\ell+k$ gives 
    \begin{equation}\label{eq-hminor-P2}
    h\geq\frac{d_0^2+3d_0-\ell-k+\varepsilon}{2} >0 .  
    \end{equation} 
    
    There are now two possibilities. 
    First, if the curves of $\mathcal{D}(p)$ do not share a common component, then one fixes $h-1$ points in $\P^2$ so that there is a pencil of curves of $\mathcal{D}(p)$ passing through these points and such that any two curves of this pencil do not share any common component. In that case, two of these curves intersect transversely in at least $\ell+k+h-1$ points, so $d_0^2\geq \ell+k+h-1$. Multiplying both sides by 2 and using Inequality (\ref{eq-hminor-P2}), one deduces that $2d_0^2\geq \ell+k+d_0^2+3d_0+\varepsilon$ and hence $\ell+k+\varepsilon \leq d_0^2 -3d_0$, which implies in particular
    \[ \ell_c-\ell_n-3k+\varepsilon \leq d_0^2-3d_0+2 ,\]
    contradicting $(ii)$.
    \newline
    
    Second, if the curves of $\mathcal{D}(p)$ all share a common component, then one can now vary $p$ in order to compute self-intersection numbers. Fix a subset $\H\subset \R\P^2$ of $h$ points in generic position and outside the curve $C$, so that for every generic $p\in\P^1$ there is exactly one curve $E_p\in\mathcal{D}(p)$ passing through these $h$ points. Now, one checks that $(E_p)_p$ is a pencil.     
    By definition of $E_p$, for all $p\in\P^1$ one has $f^{-1}(p) \subset C\cap E_p$. Write $p=[\lambda_0:\lambda_1]$ and let $P_0$ and $P_1$ be two polynomials such that $V(P_0)= E_{[0:1]}$ and $V(P_1)=E_{[1:0]}$ (both polynomials $P_1$ and $P_2$ can be taken with homogeneous coordinates in $\P^2$). The pencil $\left(C\cap V(\lambda_1P_0+\lambda_0P_1)\right)_{[\lambda_0:\lambda_1]}$ contains the pencil $(f^{-1}(p))_p$, because it contains it at the two points $p=[0:1]$ and $p=[1:0]$. In particular, the curve $V(\lambda_1P_0+\lambda_0P_1)$ passes through all the points of $f^{-1}([\lambda_0:\lambda_1])$ but also through all the fixed $h$ points, because both $V(P_0)$ and $V(P_1)$ do. As a consequence, $V(\lambda_1P_0+\lambda_0P_1) = E_{[\lambda_0:\lambda_1]}$, implying that $(E_p)_p$ is a pencil. 

    We would like now to compute $E_p^2$ by counting the number of intersection points between $E_p$ and another element $E_q$ of the pencil. However, this strategy may fail if the pencil $(E_p)_p$ has a common component. Denote by $E$ the potential common component of this pencil, and let $\H'\subset\H$ be the set of the $h'$ points of $\H$ through which $E$ does not pass. As divisors, one has $E_p=E+F_p$, where $(F_p)_p$ is a pencil without common component. We get 
    \[ E_p^2 = E^2 + 2E\cdot F_p + F_p^2 \]
    and we will bound by below each of the three terms.

    Since $(F_p)_p$ is a pencil without common component, then one can compute the self-intersection $F_p^2$ as the intersection of two elements of the pencil $F_p$ and $F_q$. Since the morphism $f:C\to\P^1$ is given by the intersection of the curve $C$ with the pencil $(F_p)_p$ it means that this pencil has at least $F_p\cdot C-\deg(f)$ base points on the curve $C$ (note that $F_p$ and $C$ generically do not share any common component because otherwise one of the component of $C$ would be a common component of the pencil $(F_p)_p$). Moreover, the curves $F_p$ intersect each of the components of $\R C$, so that $F_p\cdot C$ is at least $2\ell_c+\ell_n$ (each oval is intersected at least twice). Hence the pencil $(F_p)_p$ has at least $2\ell_c+\ell_n-\ell-k=\ell_c-k$ base points on the curve $C$. Adding the $h'$ base points in $\mathcal{H'}$ which are fixed outside of $C$, the pencil $(F_p)_p$ has at least $\ell_c-k+h'$ base points, so $F_p^2\geq \ell_c-k+h'$.

    Now, notice that $E^2 \geq h-h'-1$. In fact, the curve $E$ is determined by $h-h'$ points in generic position, so its self-intersection is at least $h-h'-1$ (the genericity ensures that by varying one of the fixed points, one gets a pencil of curves linearly equivalent to $E$, without common component and passing by the $h-h'-1$ other fixed points). Last, since $E$ and $F_p$ do not share a common component then one has $E\cdot F_p \geq 0$ and we get 
    $$ d_0^2=E_p^2\geq E^2+F_p^2\geq h-h'-1+\ell_c-k+h'=\ell_c-k+h-1.$$
    Finally, by multiplying this inequality by 2 and applying Inequality (\ref{eq-hminor-P2}), one gets 
    \[ 2d_0^2\geq 2\ell_c-2k+d_0^2+3 d_0-(\ell+k-\varepsilon)-2 \] which yields \[ d_0^2-3d_0+2 \geq \ell_c-\ell_n-3k+\varepsilon,\] contradicting $(ii)$.
\end{proof}

\begin{ex}
   Consider plane curves of degree $7$. We have $g+1=16$, hence the real part of a separating curve $C$ has an even number $\ell$ of connected components. One has $\ell_n=1$ so $\ell_c=\ell-1$. To apply Theorem \ref{thm-P2} we choose a divisor $D$ of degree 4. It decomposes as $D=2D_0+D_1$ with $D_1=0$ and $D_0$ a degree 2 divisor. The hypothesis on $(\ell,k)$ are
  \begin{enumerate}[label=$(\roman*)$]
      \item $\ell+k-\varepsilon < d_0^2+3d_0 = 10$
      \item $\ell-2-3k + \varepsilon  > d_0^2-3d_0+2 =0$.
  \end{enumerate}
   Hence for $(\ell,k) \in \{ (6,0),(6,1), (8,0),(8,1) \}$ there is no separating morphism of degree $\ell+k$ on a degree $7$ curve with $\ell$ connected components in its real part, \ie the separating gonality of a degree $7$ plane curve with $6$ (resp. $8$) connected components in its real part is at least $8$ (resp. $10$).       
   (If $\ell=2,4$ then there is no separating morphism of degree $\ell$ because the gonality of $C$ is 6.) 
\end{ex}

\begin{ex}
Consider plane curves of degree $9$. We have $g+1 = 29$, hence a separating curve $C$ has an odd number $\ell$ of connected components in its real part. One has $\ell_n =1$ and $\ell_c = \ell-1$. To apply Theorem \ref{thm-P2}, we need to choose a divisor $D$ of degree $6$.

The divisor $D$ can be decomposed as $2D_0 + D_1$, where $D_1 = 0$ and $D_0$ is a degree $d_0 = 3$ divisor. The hypotheses on $(\ell, k)$ are 
\begin{enumerate}[label=$(\roman*)$]
       \item $\ell+k-\varepsilon < d_0^2+3d_0 = 18$
       \item $\ell-2-3k + \varepsilon  > d_0^2-3d_0+2 = 2$.
   \end{enumerate}
Hence, for 
\[ (\ell, k) \in \{ (9,0), (9,1), (11,0), (11,1), (11,2), (13,0), (13,1), (13,2), (15,0),(15,1),(15,2), (17,0) \}, \] there is no separating morphism of degree $\ell+k$ from a curve of degree $9$ with $\ell$ connected components in its real part.

Another possible decomposition is $D = 2D_0+D_1$ where the degree of $D_1$ is $d_1 = 2$ and the degree of $D_0$ is $d_0 = 2$. The hypotheses on $(\ell, k)$ are 
\begin{enumerate}[label=$(\roman*)$]
       \item $\ell+k-\varepsilon < d_0^2+3d_0 = 10$
       \item $\ell-2-3k + \varepsilon  > d_0^2-3d_0+2 = 0$.
\end{enumerate}
This prohibits the pair $(7,1)$, which was not prohibited by the gonality or by the previous application of Theorem \ref{thm-P2}.

These two applications of the theorem allow one to obtain that if the real part of the curve has $15$ connected components or less, then the curve does not admit a separating morphism of degree $\ell$ or $\ell+1$.
\end{ex}

A systematic application of this observation gives the following result.

\begin{cor}\label{coro-P2}
    Let $C\subset \P^2$ be an irreducible non-singular separating curve of degree $m$ whose real part has $\ell$ connected components. Let $k\geq0$ be an integer and $\varepsilon = \ell+k \mod 2$. If 
    \begin{enumerate}
        \item[(a)] either $m$ is odd and 
        \[ k^2 +2k+3-\varepsilon \leq \ell \leq  \frac{(m-3)(m+3)}{4} -k-1+\varepsilon\]

        \item[(b)] or $m$ is even and 
        \[ k^2 + 2(1-\varepsilon)k +\varepsilon +1 \leq  \ell \leq  \frac{(m-4)(m+2)}{4} -k-1+\varepsilon\]
    \end{enumerate}
    then there exists no separating morphism $f:C\to\P^1$ of degree $\ell+k$.

    In particular, if $k_{\max}$ is the maximal $k$ such that the previous hypotheses are satisfied, then the separating gonality of $C$ is at least $\ell+k_{\max}+1$.
\end{cor}

\begin{proof}
\begin{enumerate}
\item[(a)] If $m$ is odd then $\R C$ has one pseudo-line, \ie $\ell_n=1$ and $\ell_c=\ell-1$. 
Choose $D\in|C+K_{\P^2}|$ (\ie $D$ is a curve of degree $d=m-3$) and a decomposition $D=2D_0+D_1$ as in \cite{O21}. 
By Theorem~\ref{thm-P2} there is no separating morphism of degree $\ell+k$ if
   \begin{enumerate}[label=$(\roman*)$]
       \item $\ell+k-\varepsilon < d_0^2+3d_0$
       \item $\ell-2-3k + \varepsilon  > d_0^2-3d_0+2$
   \end{enumerate}
   \ie if $\ell$ lies in the integers interval
   \[ I(d_0) = \llbracket d_0^2 - 3d_0+5+3k-\varepsilon ,  d_0^2+3d_0-k-1+\varepsilon  \rrbracket. \]
    Let $f(d_0)=d_0^2 - 3d_0+5+3k-\varepsilon$ and $g(d_0)=d_0^2+3d_0-k-1+\varepsilon$. 
 If $g(d_0-1) \geq f(d_0) -1$  then one has
    \[ I(d_0-1) \cup I(d_0) =\llbracket f(d_0-1),g(d_0) \rrbracket \]
    \ie one can glue the intervals. 
    We compute 
\begin{align*}
     g(d_0-1)- f(d_0) +1 &= (d_0-1)^2 + 3(d_0-1)-k-1 +\varepsilon - d_0^2+3d_0-5-3k+\varepsilon +1 \\
     &= 4d_0-4k-7 +2\varepsilon,
\end{align*} 
so one can glue the intervals if $d_0 \geq k+ \frac{7}{4}-\frac{\varepsilon}{2}$, \ie if $d_0 \geq k+2$. 
Now, if $k$ is fixed then $d_0$ can vary between $k+2$ and $(m-3)/2$ and we conclude that if 
\[ f(k+1)  \leq \ell \leq g\left(\frac{m-3}{2}\right) \]
then there is no separating morphism of degree $\ell+k$.
\newline

\item[(b)] If $m$ is even then $\R C$ has no pseudo-line, \ie $\ell_n=0$ and $\ell_c=\ell$. 
Choose $D\in|C+K_{\P^2}|$ (\ie $D$ is a curve of degree $d=m-3$) and a decomposition $D=2D_0+D_1$ as in \cite{O21}. 
By Theorem~\ref{thm-P2} there is no separating morphism of degree $\ell+k$ if
   \begin{enumerate}[label=$(\roman*)$]
       \item $\ell+k-\varepsilon < d_0^2+3d_0$
       \item $\ell-3k + \varepsilon  > d_0^2-3d_0+2$
   \end{enumerate}
   \ie if $\ell$ lies in the integers interval
   \[ I(d_0) = \llbracket d_0^2 - 3d_0+3+3k-\varepsilon ,  d_0^2+3d_0-k-1+\varepsilon  \rrbracket. \]
    Let $f(d_0)=d_0^2 - 3d_0+3+3k-\varepsilon$ and $g(d_0)=d_0^2+3d_0-k-1+\varepsilon$. 
 If $g(d_0-1) \geq f(d_0) \gadd{-1}$ then one has
    \[ I(d_0-1) \cup I(d_0) =\llbracket f(d_0-1),g(d_0) \rrbracket \]
    \ie one can glue the intervals. 
    We compute 
\begin{align*}
     g(d_0-1)- f(d_0)+1 &= (d_0-1)^2 + 3(d_0-1)-k-1 +\varepsilon - d_0^2+3d_0-3-3k+\varepsilon +1 \\
     &= 4d_0-4k-5 +2\varepsilon,
\end{align*} 
so one can glue the intervals if $d_0 \geq k+ \frac{5}{4}-\frac{\varepsilon}{2}$, \ie if $d_0 \geq k+2$ (when $\ell+k$ is even) and if $d_0 \geq k+1$ (when $\ell+k$ is odd), \ie if $d_0 \geq k+2-\varepsilon$. 
Now, if $k$ is fixed then $d_0$ can vary between $k+2-\varepsilon$ and $(m-4)/2$ and we conclude that if 
\[ f(k+1-\varepsilon)  \leq \ell \leq g\left(\frac{m-4}{2}\right) \]
then there is no separating morphism of degree $\ell+k$.
\newline
\end{enumerate}

The last point comes from the fact that if $k\geq1$ satisfies the hypotheses, then so does $k-1$. Hence $C$ does not admit any separating morphism of degree $\ell+k$ for $0\leq k \leq k_{\max}$, meaning its separating gonality is at least $\ell+k_{\max}+1$.
\end{proof}

\subsection{For Hirzebruch surfaces $\F_n$}

In $\F_n$, a curve of bidegree $(a,b)$ is a curve of class $aE_0+bF$, where $E_0^2=n$ and $F^2=0$, see Figure \ref{fig-trapeze}. Let $E_\infty$ be the curve class of bidegree $(1,-n)$. The canonical divisor $K_{\F_n}$ is of bidegree $(-2,-2+n)$.

\begin{figure}[h]
\centering
\begin{tikzpicture}

    \draw (0,0) -- (0,2) -- (3,2) -- (7,0) -- cycle ;
    
    \node at (3,0.3) {$E_0$} ;
    \node at (0.3,1) {$F$} ;
    \node at (1.5,1.7) {$E_\infty$} ;

    \draw[<->] (0,-0.2) -- (7,-0.2) ;
    \node at (3,-0.6) {$an+b$} ;
    \draw[<->] (-0.2,0) -- (-0.2,2) ;
    \node at (-0.6,1) {$a$} ;
    \draw[<->] (0,2.2) -- (3,2.2) ;
    \node at (1.5,2.6) {$b$} ;

\end{tikzpicture}
\caption{The Newton polygon of a curve of bidegree $(a,b)$ in $\F_n$.}
\label{fig-trapeze}
\end{figure}
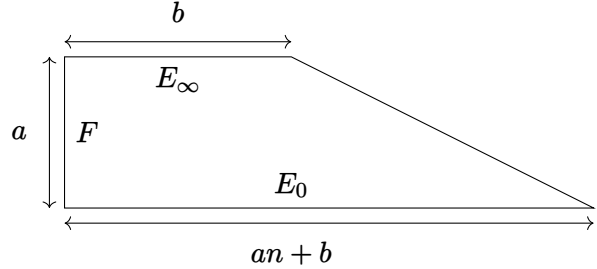

\begin{thm}\label{thm-P1P1}
    Let $C\subset \F_n$ be an irreducible non-singular separating curve of bidegree $(a,b)$ whose real part has $\ell$ connected components. Choose $D\in|C+K_{\F_n}|$ and a decomposition $D=2D_0+D_1$ such that $D_0$ is effective of bidegree $(d_0^x,d_0^y)$ and with $d_0^x\geq1$ and $d_0^y\geq 1$, and $D_1$ is a reduced curve which does not contain $E_\infty$. Assume furthermore that $d_0^x\geq1$ and $d_0^y \geq1$. Let $\ell_c$ be the number of ovals of $\R C$ and $\ell_n=\ell-\ell_c$. Let $k\geq0$ be an integer and $\varepsilon = \ell+k \mod 2$. If
    \begin{enumerate}[label=(\roman*)]
        \item $\ell+k -\varepsilon < 2(d_0^xd_0^y+d_0^x+d_0^y) + nd_0^x(d_0^x+1)$  
        \item $\ell_c-\ell_n-3k + \varepsilon > 2(d_0^x-1)(d_0^y-1)+nd_0^x(d_0^x-1) $ 
    \end{enumerate}
    then there exists no separating morphism $f:C\to\P^1$ of degree $\ell+k$.
\end{thm}

\begin{proof}
    Assume that there exists a separating morphism $f:C\to\P^1$ of degree $\ell+k$ satisfying~$(i)$, \ie such that $\ell+k-\varepsilon <  2(d_0^xd_0^y+d_0^x+d_0^y) + nd_0^x(d_0^x+1)$. We will show that $(ii)$ cannot be satisfied.
    
    Choose a point $p\in\RP^1$ and consider the linear system $\mathcal{D}(p)\subset|D_0|$ of effective divisors linearly equivalent to $D_0$ and passing through all the points of $f^{-1}(p)$. Since $D_1$ does not contain $E_\infty$ then by Lemma \ref{lem-dimD(p)}, the dimension $h$ of $\mathcal{D}(p)$ is at least $\dim| D_0|-\frac{\ell+k-\varepsilon}{2}$. 
    One has $\dim|D_0| = d_0^xd_0^y +d_0^x+d_0^y + \frac{n}{2}d_0^x(d_0^x+1)$ so hypothesis~$(i)$ on $\ell+k$ gives 
    \begin{equation}\label{eq-hminor-Fn}
    h\geq d_0^xd_0^y +d_0^x+d_0^y + \frac{n}{2}d_0^x(d_0^x+1) - \frac{\ell+k-\varepsilon}{2} >0 .  
    \end{equation}
    
    There are now two possibilities. 
    First, if the curves of $\mathcal{D}(p)$ do not share a common component, then one fixes $h-1$ points in $\F_n$ so that there is a pencil of curves of $\mathcal{D}(p)$ passing through these points and such that any two curves of this pencil do not share any common component. In that case, two of these curves intersect transversely in at least $\ell+k+h-1$ points, so $D_0^2\geq \ell+k+h-1$. Multiplying both sides by 2 and using Inequality~(\ref{eq-hminor-Fn}) one deduces that $2D_0^2\geq \ell+k+D_0^2-K_{\F_n}D_0+\varepsilon$ and hence $\ell+k+\varepsilon \leq D_0^2 + K_{\F_n}D_0$, which implies in particular
    \[ \ell_c-\ell_n-3k + \varepsilon \leq D_0^2+K_{\F_n}D_0+2 =  2(d_0^x-1)(d_0^y-1)+nd_0^x(d_0^x-1),\]
    meaning $(ii)$ is not satisfied. In this case, note that we actually proved a stronger statement: we could replace condition $(ii)$ by the weaker version 
    \begin{enumerate}
        \item[$(ii)'$] $\ell+k+\varepsilon >  2(d_0^x-1)(d_0^y-1)+nd_0^x(d_0^x-1)$. 
    \end{enumerate}
    The need for Condition $(ii)$ arises when studying the second case.
    \newline
    
    Second, if the curves of $\mathcal{D}(p)$ all share a common component, then one can now vary $p$ in order to compute self-intersection numbers. Fix a subset $\H\subset \R \F_n$ of $h$ points in generic position and outside the curve $C$, so that for every generic $p\in\P^1$ there is exactly one curve $E_p\in\mathcal{D}(p)$ passing through these $h$ points. Now, one checks that $(E_p)_p$ is a pencil.  
    By definition of $E_p$, for all $p\in\P^1$ one has $f^{-1}(p) \subset C\cap E_p$. Write $p=[\lambda_0:\lambda_1]$ and let $P_0$ and $P_1$ be two polynomials such that $V(P_0)= E_{[0:1]}$ and $V(P_1)=E_{[1:0]}$ (in a local chart of $\F_n$). The pencil $\left(C\cap V(\lambda_1P_0+\lambda_0P_1)\right)_{[\lambda_0:\lambda_1]}$ contains the pencil $(f^{-1}(p))_p$, because it contains it at the two points $p=[0:1]$ and $p=[1:0]$. In particular, the curve $V(\lambda_1P_0+\lambda_0P_1)$ passes through all the points of $f^{-1}([\lambda_0:\lambda_1])$ but also through all the fixed $h$ points, because both $V(P_0)$ and $V(P_1)$ do. As a consequence, $V(\lambda_1P_0+\lambda_0P_1) = E_{[\lambda_0:\lambda_1]}$, implying that $(E_p)_p$ is a pencil.

    We would like now to compute $E_p^2$ by counting the number of intersection points between $E_p$ and another element $E_q$ of the pencil. However, this strategy may fail if the pencil $(E_p)_p$ has a common component. Denote by $E$ the potential common component of this pencil, and $\H'\subset\H$ be the set of the $h'$ points of $\H$ through which $E$ does not pass. As divisors one has $E_p=E+F_p$, where $(F_p)_p$ is a pencil without common component. We get 
    \[ E_p^2 = E^2 + 2E\cdot F_p + F_p^2 \]
    and we will bound by below each of the three terms.
    
    Since $(F_p)_p$ is a pencil without common component, then one can compute its self-intersection $F_p^2$ as the intersection of two elements of the pencil $F_p$ and $F_q$. The morphism $f:C\to\P^1$ is given by the intersection of the curve $C$ with the pencil $(F_p)_p$, implying that this pencil has at least $F_p\cdot C-\deg(f)$ base points on the curve $C$ (note that $F_p$ and $C$ do not share any common component because otherwise $C$ would be a common component of $F_p$). Moreover, the curves $F_p$ intersect each of the components of $\R C$, so that $F_p\cdot C$ is at least $2\ell_c+\ell_n$ (each oval is intersected at least twice). Hence the pencil $(F_p)_p$ has at least $2\ell_c+\ell_n-\ell-k=\ell_c-k$ base points on the curve $C$. Adding the $h'$ base points in $\mathcal{H'}$ which are fixed outside of $C$, the pencil $(F_p)_p$ has at least $\ell_c-k+h'$ base points, so $F_p^2\geq \ell_c-k+h'$.

    The curve $E$ is determined by $h-h'$ points in generic position so it is non-singular. Assume $E$ has bidegree $(d^x,d^y)$. Two cases can happen.
    \begin{itemize}
        \item If $n \geq 1$ and $d^x$ is non-zero or if $n=0$ and both $d^x$ and $d^y$ are non-zero, then $E$ is irreducible, otherwise it would be singular. Then, by varying one of the fixed points one gets a pencil of curves linearly equivalent to $E$ without common component, with $h-h'-1$ base points. Hence one has $E^2\geq h-h'-1$. Note that $E\cdot F_p \geq0$.

        \item Assume now that $E = (0, d^y)$; the case $n=0$ and $E = (d^x, 0)$ can be treated similarly. One has $E^2 = 0$ and we will compute $E\cdot F_p$. 
        
        First, note that $d^y=h-h'$. Indeed, each of the $h-h'$ points must be on one of the $d^y$ components of $E$, and two of these points cannot be on the same component otherwise the configuration is not generic, so $h-h' \leq d^y$. But any component of $E$ must contain a point, otherwise this component can move and we get a contradiction, so $d^y\leq h-h'$.

        Now, since $F_p = E_p-E \sim D_0-E$ then $F_p$ has bidegree $(d_0^x,d_0^y-d^y)$ and we compute 
        \[  E\cdot F_p = d_0^x d^y \geq h-h' \]
        where we used $d_0^x\geq1$. In particular, we get $2E\cdot F_p \geq h-h'-1$.

    \end{itemize}

    In both cases, one has
    \[ D_0^2 = E_p^2 = E^2 + 2E\cdot F_p + F_p^2 \geq (\ell_c-k+h') + (h-h'-1)+0 = \ell_c-k+h-1.  \]
    Finally, by multiplying this inequality by 2 and applying Inequality (\ref{eq-hminor-Fn}) one gets 
    \[ 2D_0^2\geq 2\ell_c-2k+ 2(d_0^xd_0^y +d_0^x+d_0^y) + nd_0^x(d_0^x+1) - \ell-k+\varepsilon -2 \]
    which yields 
    \[ 2(d_0^x-1)(d_0^y-1)+nd_0^x(d_0^x-1) \geq \ell_c-\ell_n-3k+\varepsilon,\]
    contradicting condition $(ii)$.
\end{proof}

\begin{ex}
Let $C \subset \mathbb \P_1 \times \P_1$ be an irreducible real curve of bidegree $(4,4)$ whose real scheme is $\langle2\langle1\rangle\rangle$, see Figure \ref{fig-exP1P1}. 
Note that this curve is separating. Indeed, the pencil of planes generated by the line $d$ on Figure \ref{fig-exP1P1} defines a real and separating morphism of degree~$8$, corresponding to the $8$ intersection points of $C$ with a plane. 
If the line $d$ intersects one of the two inner ovals at a point $p$, then we obtain a separating morphim of degree $7$, the image of $p$ being the parameter $t\in\P^1$ corresponding to the only plane of the pencil that is tangent to the oval (\ie has a single point of intersection with the oval).
Similarly, if $d$ intersects both inner ovals, we obtain a separating morphism of degree $6$.

\begin{figure}[h!]
    \centering
    \includegraphics[scale=0.25]{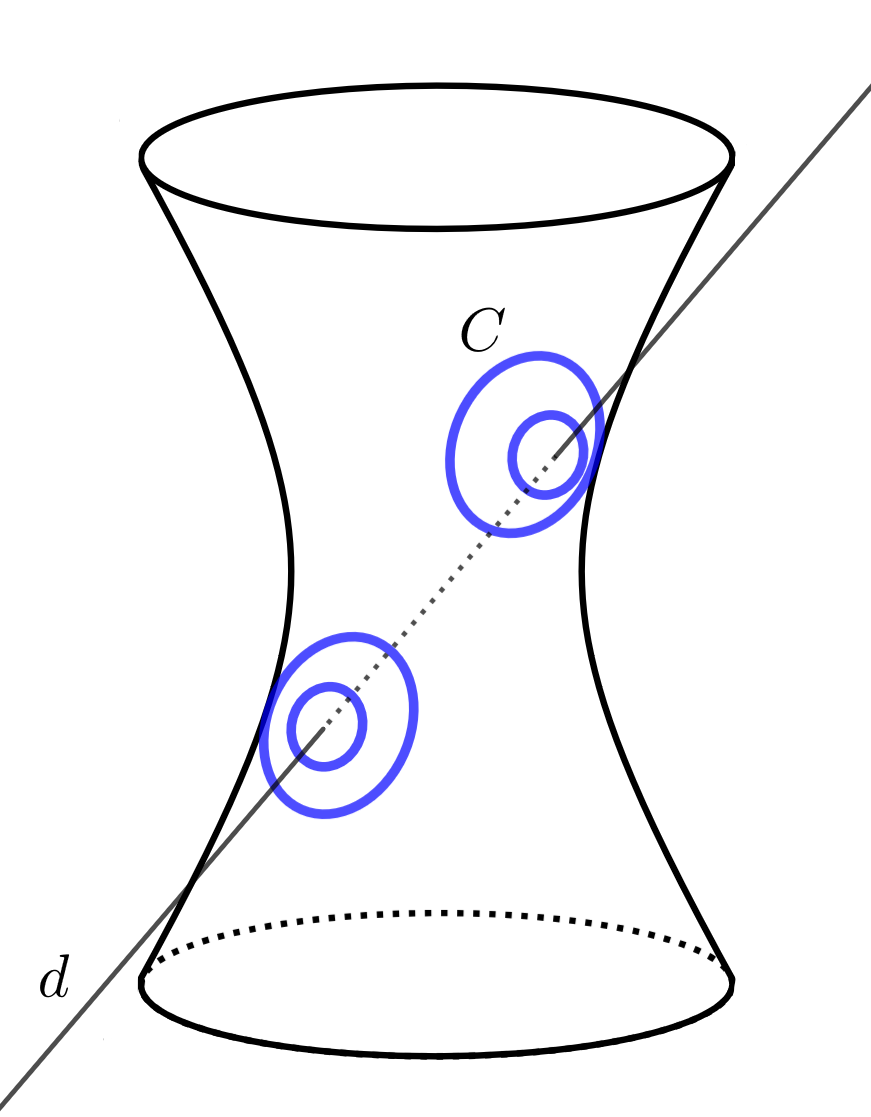}
    \caption{A curve of bidegree $(4,4)$ with real scheme $\langle2\langle1\rangle\rangle$.}
    \label{fig-exP1P1}
\end{figure}

We now apply Theorem \ref{thm-P1P1} to $C$. One has $\ell=\ell_c=4$ and $\ell_n=0$. One can choose $k=0$ or $1$, $D_0$ to be of bidegree $(1,1)$ and $D_1$ to be empty. One has
    \begin{enumerate}[label=$(\roman*)$]
        \item $4 = \ell +k-\varepsilon < 2(d_0^xd_0^y+d_0^x+d_0^y) + 0 = 6$ 
        \item $4 \text{ or } 2 = \ell_c-\ell_n -3k + \varepsilon > 2(d_0^x-1)(d_0^y-1)+ 0 = 0$
    \end{enumerate}
so $C$ does not admit any separating morphim of degree $4$ and $5$. We conclude that the separating gonality of $C$ is $6$.
\end{ex}

\begin{ex}
Let $C \subset \mathbb \P_1 \times \P_1$ be an irreducible real curve of bidegree $(6,6)$ whose real scheme is $\langle2\langle1\langle1\rangle\rangle\rangle$. Similarly to the previous example, this curve is separating and there exists a separating morphism of degree $10$. Application of Theorem \ref{thm-P1P1} with $D_0$ of bidgree $(2,2)$ shows there is no separating morphism of degree $6$ or $7$, hence the separating gonality is $8$, $9$ or $10$.
\end{ex}

\begin{rmk}
    The beginning of the proofs of Theorems \ref{thm-P2} and \ref{thm-P1P1} are the same : 
    lower bound for the dimension of $\mathcal{D}(p)$, case where the curves of $\mathcal{D}(p)$ do not share a common component, introduction of $(E_p)_p$, $E$ and $(F_p)_p$ when there is a common component, and lower bound for $F_p^2$.
    We then want to find a lower bound for $E^2+E\cdot F_p$, and the reasoning applied to achieve this last step is specific to the surface we consider, either $\P^2$ or $\F_n$. 

    We believe that a similar result should hold for many surfaces $X$ under the hypotheses that
    \begin{itemize}
        \item $\chi(\O_X) = 1$ and $H^2(X,\mathcal{L}(D_0))=0$ (to use Riemann--Roch theorem and obtain $\dim|D_0|\geq  \frac{D_0^2-K_XD_0}{2}$, which is what we use directly in the case of $\P^2$ and $\F_n$),
    \item    \begin{enumerate}[label=$(\roman*)$]
        \item $\ell+k -\varepsilon < D_0^2-K_X\cdot D_0$
        \item $\ell_c-\ell_n-3k + \varepsilon > D_0^2 + K_X\cdot D_0 +2 = 2g(D_0)$.
    \end{enumerate}
    \end{itemize} 
    However, we are not able to carry out the strategy in this more general context.
\end{rmk}

\printbibliography 
\end{document}